\documentclass[11pt]{article}
\usepackage[margin=1in]{geometry}
\usepackage{amsmath,amssymb,amsthm,amsfonts,verbatim}
\usepackage{calc}
\usepackage{array}
\usepackage{graphicx}
\usepackage{enumerate}
\usepackage{booktabs}
\usepackage{url}

\title{A stability conjecture for the unstable cohomology\\ of $\SL_n\Z$, mapping class groups, and $\Aut(F_n)$}

\author{Thomas Church, Benson Farb and Andrew Putman\thanks{The authors gratefully acknowledge support from the National Science Foundation.}}

\theoremstyle{plain}
\newtheorem{conjecture}{Conjecture}
\newtheorem*{conjmcgre}{Conjecture \ref{conjecture:mcg}, restated}
\newtheorem*{conjslnre}{Conjecture \ref{conjecture:SLn}, restated}
\newtheorem{theorem}[conjecture]{Theorem}
\newtheorem{proposition}[conjecture]{Proposition}
\theoremstyle{definition}
\newtheorem{definition}[conjecture]{Definition}
\newtheorem*{remark}{Remark}
\newtheorem{remarknum}[conjecture]{Remark}
\newtheorem{question}[conjecture]{Question}
\newtheorem{problem}[conjecture]{Problem}

\newcommand{\nc}{\newcommand}
\nc{\dmo}{\DeclareMathOperator}

\nc{\Q}{\mathbb{Q}}
\nc{\R}{\mathbb{R}}
\nc{\Z}{\mathbb{Z}}
\nc{\C}{\mathcal{C}}
\nc{\Cpx}{\mathbb{C}}
\nc{\N}{\mathbb{N}}
\nc{\B}{\mathcal{B}}
\dmo{\Out}{Out}
\dmo{\Aut}{Aut}
\dmo{\Hom}{Hom}
\dmo{\Stab}{Stab}
\dmo{\kernel}{kernel}
\dmo\im{im}
\dmo\id{id}
\dmo\SL{SL}
\dmo\SO{SO}
\dmo\SU{SU}
\dmo\Sp{Sp}
\dmo\Mod{Mod}
\dmo\PMod{PMod}
\nc{\M}{\mathcal{M}}
\nc{\T}{\mathcal{T}}
\dmo{\thick}{thick}
\dmo{\vcd}{vcd}
\dmo{\cd}{cd}
\dmo{\St}{St}
\dmo{\spn}{span}

\nc\so{\mathfrak{so}}
\nc{\bwedge}{\textstyle{\bigwedge}}

\nc\vect[1]{#1}

\renewcommand{\phi}{\varphi}

\renewcommand{\epsilon}{\varepsilon}
\nc{\coloneq}{\mathrel{\mathop:}\mkern-1.2mu=}
\nc{\margin}[1]{\marginpar{\scriptsize #1}}
\nc{\para}[1]{\medskip\noindent\textbf{#1.}}

\newdimen{\QMod}
\newdimen{\QSq}
\newdimen{\QAut}
\begin{document}

\maketitle

\begin{abstract}
In this paper we conjecture the stability and vanishing of a large piece of the unstable rational cohomology of $\SL_n\Z$, of mapping class groups, and of $\Aut(F_n)$. 
\end{abstract}

\section{Introduction} For each of the sequences of groups in the title, the $i$-th rational cohomology is known to be independent of $n$ in a linear range $n\geq Ci$. Furthermore, this ``stable cohomology'' has been explicitly computed in each case. In contrast, very little is known about the unstable cohomology. In this paper we conjecture a new kind of stability in the cohomology of these groups.
These conjectures concern the unstable cohomology, in a range near the ``top dimension'', and in the first two cases imply vanishing of the unstable cohomology in this range. 
 

\section{Stability in the unstable cohomology of $\SL_n\Z$} 
\label{section:SLn}
The rational cohomology of the arithmetic group $\SL_n\Z$ coincides with that of the associated locally symmetric space $X_n\coloneq \SL_n\Z\backslash \SL_n\R/\SO(n)$:
\begin{equation}
\label{eq:locsym}
H^i(\SL_n\Z;\Q) \approx H^i(X_n;\Q).
\end{equation}
Borel \cite{Borel} proved that for each $i\geq 0$ the group $H^i(\SL_n\Z;\Q)$ does not depend on $n$ for $n\gg i$; it is believed that the optimal stable range should be $n>i+1$, though this has not been proved. Borel--Serre proved \cite{BS} that the \emph{virtual cohomological dimension} of $\SL_n\Z$ is \[\vcd(\SL_n\Z)=\binom{n}{2}.\]
This implies that $H^k(\SL_n\Z;\Q)=0$ for all $k>\binom{n}{2}$.

\begin{conjecture}[{\bf Stable instability}]
\label{conjecture:SLn}
For each $i\geq 0$ the group $H^{\binom{n}{2}-i}(\SL_n\Z;\Q)$ does not depend on $n$ for $n> i+1$.
\end{conjecture}

The form of Conjecture~\ref{conjecture:SLn} may appear surprising to readers familiar with the known examples of homological stability, so we now explain some of the intuition behind the conjecture. The locally symmetric space $X_n$ is an orbifold of dimension $\binom{n+1}{2}-1$. Thus if $X_n$ were compact, Poincar\'e duality combined with \eqref{eq:locsym} and Borel's stability theorem would imply that $H^{\binom{n+1}{2}-1-i}(\SL_n\Z;\Q)$ was independent of $n$ for $n\gg i$. However $X_n$ is not compact and does not satisfy Poincar\'e duality. The more general notion of Bieri--Eckmann duality allows us to repair this gap, and also lets us give in \eqref{eq:SLnsteinstab} one concrete approach to proving Conjecture~\ref{conjecture:SLn}.

However, this approach to Conjecture~\ref{conjecture:SLn} has the peculiar consequence that if it  holds, then in fact the unstable cohomology \emph{vanishes} in the range of stability, as we will explain in detail below.
\begin{conjecture}[{\bf Vanishing Conjecture}]
\label{conjecture:SLnvanishing}
$H^{\binom{n}{2}-i}(\SL_n\Z;\Q)=0$ for all $i<n-1$.
\end{conjecture}
For $i=0$, Conjecture~\ref{conjecture:SLnvanishing} is a theorem of Lee--Szczarba \cite{LS}, who proved that $H^{\binom{n}{2}}(\SL_n\Z;\Q)=0$ for all $n\geq 2$. We prove Conjecture~\ref{conjecture:SLnvanishing} for $i=1$ in \cite{CFP2}; this case can also be deduced from Bykovskii \cite[Theorem 2]{By}.
We will revisit the connection between Conjecture~\ref{conjecture:SLn} and Conjecture~\ref{conjecture:SLnvanishing} after describing the maps realizing the stability proposed in Conjecture~\ref{conjecture:SLn}.

\para{Computational evidence}
The rational cohomology groups of $\SL_n\Z$ have been completely computed for $2 \leq n \leq 7$.  These
calculations are summarized in Table~\ref{table:cohomology}.  The data in this table
is in agreement with Conjecture~\ref{conjecture:SLnvanishing}.
\begin{table}
\begin{centering}
\begin{tabular}{l@{\hspace{22pt}}l@{\hspace{22pt}}p{\QSq}@{\hspace{5pt}}p{\QSq}@{\hspace{5pt}}p{\QSq}@{\hspace{5pt}}p{\QSq}@{\hspace{5pt}}p{\QSq}@{\hspace{5pt}}p{\QSq}@{\hspace{5pt}}p{\QSq}@{\hspace{5pt}}p{\QSq}@{\hspace{5pt}}p{\QSq}@{\hspace{5pt}}p{\QSq}@{\hspace{5pt}}p{\QSq}@{\hspace{5pt}}p{\QSq}@{\hspace{5pt}}p{\QSq}@{\hspace{5pt}}p{\QSq}@{\hspace{5pt}}p{\QSq}@{\hspace{5pt}}p{\QSq}@{\hspace{5pt}}p{\QSq}@{\hspace{5pt}}p{\QSq}@{\hspace{5pt}}p{\QSq}@{\hspace{5pt}}p{\QSq}@{\hspace{5pt}}p{\QSq}@{}}
\toprule
$n$ & \!\!\!$\vcd$ & \multicolumn{21}{l}{\!\!\!\!$H^1(\SL_n\Z;\Q)$ \hspace{5pt} $H^2(\SL_n\Z;\Q)$ \hspace{5pt}  $\cdots$ \hspace{5pt} $H^{\vcd}(\SL_n\Z;\Q)$}      \\
\midrule
$2$ & $1$  & $0$ &     &      &     &        &     &     &      &      &      &      &     &     &      &      &     &     &     &     &     &     \\
$3$ & $3$  & $0$ & $0$ & $0$  &     &        &     &     &      &      &      &      &     &     &      &      &     &     &     &     &     &     \\
$4$ & $6$  & $0$ & $0$ & $\Q$ & $0$ & $0$    & $0$ &     &      &      &      &      &     &     &      &      &     &     &     &     &     &     \\
$5$ & $10$ & $0$ & $0$ & $0$  & $0$ & $\Q$   & $0$ & $0$ & $0$  & $0$  & $0$  &      &     &     &      &      &     &     &     &     &     &     \\
$6$ & $15$ & $0$ & $0$ & $0$  & $0$ & $\Q^2$ & $0$ & $0$ & $\Q$ & $\Q$ & $\Q$ & $0$  & $0$ & $0$ & $0$  & $0$  &     &     &     &     &     &     \\
$7$ & $21$ & $0$ & $0$ & $0$  & $0$ & $\Q$   & $0$ & $0$ & $0$  & $\Q$  & $0$  & $0$ & $0$ & $0$ & $\Q$ & $\Q$ & $0$ & $0$ & $0$ & $0$ & $0$ & $0$ \\
\bottomrule
\end{tabular}
\caption{The rational cohomology of $\SL_n\Z$ for $2 \leq n \leq 7$.  For $n=2$ this is classical;
for $n=3$ this was calculated by Soul\'e \cite{Soule}; for $n=4$, by 
Lee--Szczarba \cite{LeeSzczarbaTorsion}; and for $5 \leq n \leq 7$, by
Elbaz-Vincent--Gangl--Soul\'e \cite{ElbazVincentGanglSoule}. The classes in $H^3(\SL_4\Z;\Q)$, $H^8(\SL_6\Z;\Q)$, $H^{10}(\SL_6\Z;\Q)$, and $H^{15}(\SL_7\Z;\Q)$, as well as one dimension in $H^5(\SL_6\Z;\Q)$, are~unstable.}
\label{table:cohomology}
\end{centering}
\end{table}

\para{Possible approaches} An important feature of Conjecture~\ref{conjecture:SLn} is that there are
natural candidates for ``stabilization maps'' between $H^{\binom{n}{2}-i}(\SL_n\Z;\Q)$ and $H^{\binom{n+1}{2}-i}(\SL_{n+1}\Z;\Q)$ which could realize the  isomorphisms conjectured in Conjecture~\ref{conjecture:SLn}.\pagebreak

\para{Parabolic stabilization} We first give a topological construction of a stabilization map 
\begin{equation*}
H^{\binom{n+1}{2}-i}(\SL_{n+1}\Z;\Q) \rightarrow H^{\binom{n}{2}-i}(\SL_n\Z;\Q)
\end{equation*}
as follows.     The stabilizer in $\SL_{n+1}\Z$ of the subspace $\Q^n<\Q^{n+1}$ is isomorphic to the semi-direct product $\Z^n\rtimes \SL_n\Z$, where the normal subgroup 
$\Z^n$ consists of those automorphisms that restrict to the identity on $\Q^n$.   Note that the action of $\SL_n\Z$ on $\Z^n$ in this semi-direct product is the standard one; in particular $\SL_n\Z$ acts trivially on $H_n(\Z^n;\Z)$.  The extension
\[1\to \Z^n\to \Z^n\rtimes \SL_n\Z\to \SL_n\Z\to 1\] 
therefore yields a Gysin map $H^k(\Z^n\rtimes \SL_n\Z;\Z)\to H^{k-n}(\SL_n\Z;\Z)$.  Taking $k=\binom{n+1}{2}-i$ and passing to rational cohomology
yields the composition 
\begin{equation}
\label{eq:SLnstabmap}
H^{\binom{n+1}{2}-i}(\SL_{n+1}\Z;\Q)\to H^{\binom{n+1}{2}-i}(\Z^n\rtimes \SL_{n+1}\Z;\Q)\to H^{\binom{n}{2}-i}(\SL_n\Z;\Q)
\end{equation}
where the first map is restriction.  We conjecture that \eqref{eq:SLnstabmap} is an isomorphism for $n> i+1$, making explicit the stabilization in Conjecture~\ref{conjecture:SLn}.

\begin{remark}Iterating this process starting with $\SL_1\Z$ yields the group of strictly upper-triangular matrices $N_n$. This is the fundamental group of an $\binom{n}{2}$-dimensional nil-manifold, and thus provides an explicit witness for the lower bound $\vcd(\SL_n\Z)\geq \binom{n}{2}$, as follows. An elementary argument (sometimes called Shapiro's Lemma) shows that
\[H^*(\SL_n\Z;M)\approx H^*(N_n;\Q)\]
when $M\coloneq \Hom_{\Q N_n}(\Q \SL_n\Z, \Q)$. We thus have $H^{\binom{n}{2}}(\SL_n\Z;M)\approx H^{\binom{n}{2}}(N_n;\Q)\approx \Q$, demonstrating that $\vcd(\SL_n\Z)\geq \binom{n}{2}$. Despite this, it follows from \cite{LS} that the fundamental class $[N_n]\in H_{\binom{n}{2}}(\SL_n\Z;\Q)$ of this $\binom{n}{2}$-manifold is trivial in the rational homology of $\SL_n\Z$.
\end{remark}

\para{Duality groups} Our second approach to Conjecture~\ref{conjecture:SLn} would give a map in the other direction, namely a map
\[H^{\binom{n}{2}-i}(\SL_n\Z;\Q)\to H^{\binom{n+1}{2}-i}(\SL_{n+1}\Z;\Q).\]
Recall that a group $\Gamma$ is a {\em duality group} if there is an integer $\nu$ and a $\Z\Gamma$-module $D$, called the {\em dualizing module} for $\Gamma$, with the property that there are isomorphisms 
\[H^{\nu-i}(\Gamma;M)\approx H_{i}(\Gamma;M\otimes_{\Z} D)\]
for any $\Z\Gamma$-module $M$.  No group which contains torsion can be a duality group.  To remedy this, we say that a group $\Gamma$ is a {\em virtual duality group} if it has some finite index subgroup which is a duality group. This implies that there exists a {\em rational dualizing $\Q\Gamma$-module} $D$ so that 
\[H^{\nu-i}(\Gamma;M)\approx H_{i}(\Gamma;M\otimes_{\Q} D)\]
for any $\Q\Gamma$-module $M$.  The integer $\nu$ equals the virtual cohomological dimension 
$\vcd(\Gamma)$. See \cite{BE} or \cite[VIII.10]{Bro} for details.

\para{Duality for \boldmath$\SL_n\Z$}
The {\em spherical Tits building} $\B(\Q^n)$ is the complex of flags of nontrivial proper subspaces of $\Q^n$. By the Solomon--Tits Theorem, $\B(\Q^n)$ is homotopy equivalent to an infinite wedge of $(n-2)$-dimensional spheres.
The {\em Steinberg module} of $\SL_n\Z$ is defined to be 
\[\St(\SL_n\Z)\coloneq H_{n-2}(\B(\Q^n);\Q).\]
Since $\SL_n\Z$ acts on $\B(\Q^n)$ by simplicial automorphisms, $\St(\SL_n\Z)$ is a $\Q\SL_n\Z$-module.  
Borel--Serre  \cite[Theorem 11.4.2]{BS} proved that  $\SL_n\Z$ is a virtual duality group with dualizing module $\St(\SL_n\Z)$ and $\nu=\vcd(\SL_n\Z)=\binom{n}{2}$, so we have natural isomorphisms \[H^{\binom{n}{2}-i}(\SL_n\Z;\Q) \cong H_i(\SL_n\Z;\St(\SL_n\Z)).\]
Given this, Conjecture~\ref{conjecture:SLn} has the following equivalent restatement.

\begin{conjslnre}
For each $i\geq 0$, the group $H_i(\SL_n\Z;\St(\SL_n\Z))$ does not depend on $n$ for $n > i+1$.
\end{conjslnre}

In this form the conjecture looks like a standard formulation of homological stability.  However, the devil is in the details of the coefficient module $\St(\SL_n\Z)$, which itself is changing with $n$.

\begin{remark}
Dwyer \cite{DwyerTwisted} (see also van der Kallen \cite{vdK}) proved that the homology of $\SL_n\Z$ stabilizes
with respect to families of twisted coefficient systems satisfying certain growth conditions.  However, the coefficient systems $\St(\SL_n\Z)$ do not satisfy Dwyer's condition.
\end{remark}

\para{Steinberg stabilization}
We now construct an explicit candidate for a stabilization map
\[H_i(\SL_n\Z;\St(\SL_n\Z))\to H_{i}(\SL_{n+1}\Z;\St(\SL_{n+1}\Z)).\]
Choose any line $L$ in $\Q^{n+1}$ such that $L_\Z\coloneq L\cap \Z^{n+1}$ defines a splitting $\Z^{n+1}=\Z^n\oplus L_\Z$. This splitting determines an inclusion $\SL_n\Z\hookrightarrow \SL_{n+1}\Z$ as the subgroup stabilizing $\Q^n$ and acting trivially on $L$; in appropriate coordinates this inclusion has the form $A\mapsto \biggl(\begin{matrix}A&0\\0&1\end{matrix}\biggr)$. To define the desired map on homology we need to construct an $\SL_n\Z$-equivariant map
\begin{equation}
\label{eq:steinmap}
\phi\colon \St(\SL_n\Z)\to \St(\SL_{n+1}\Z).
\end{equation}

 We can construct the map $\phi$ by hand. We will use the line $L$ to define a embedding $F$ of the suspension $S(\B(\Q^n))$ into $\B(\Q^{n+1})$. We describe $F$ as a map \[F\colon [0,1]\times \B(\Q^n)\to \B(\Q^{n+1})\] with the property that $\{0\}\times \B(\Q^n)$ maps to the vertex $\Q^n$ (which is indeed a proper subspace of $\Q^{n+1}$) and $\{1\}\times \B(\Q^n)$ maps to the vertex $L$.

On $\{\frac{1}{2}\}\times \B(\Q^n)$, we define the map $F$ to be the natural inclusion of $\B(\Q^n)$ into $\B(\Q^{n+1})$ determined by considering subspaces of $\Q^n$ as subspaces of $\Q^{n+1}$. We can extend $F$ across $[0,\frac{1}{2}]\times \B(\Q^n)$ by linear interpolation, since the image of $\{\frac{1}{2}\}\times \B(\Q^n)$ lies inside the star of the vertex $\Q^n$ (indeed, it is precisely the star of $\Q^n$). Explicitly, every $d$-simplex determined by a chain $0\lneq V_0\lneq V_1\lneq \cdots\lneq V_d\lneq \Q^{n+1}$ satisfying $V_d\lneq \Q^n$ sits inside the $d+1$-simplex determined by the chain $0\lneq V_0\lneq V_1\lneq \cdots\lneq V_d\lneq \Q^n\leq \Q^{n+1}$.

Finally, on $\{\frac{3}{4}\}\times \B(\Q^n)$, we define the map $F$ to take the $d$-simplex determined by a chain $0\lneq V_0\lneq V_1\lneq \cdots\lneq V_d\lneq \Q^n$ to the $d$-simplex determined by the chain \begin{equation}
\label{eq:Lsimplex}
0\ \lneq\  V_0+L\ \lneq\ V_1+L\ \lneq\ \cdots\ \lneq\ V_d+L\ \lneq\ \Q^{n+1}.
\end{equation}
We can canonically extend $F$ across $[\frac{1}{2},\frac{3}{4}]\times \B(\Q^n)$, since the convex hull of the $2d+2$ vertices $\{V_0,\ldots,V_d,V_0+L,\ldots,V_d+L\}$ is isomorphic to the standard simplicial triangulation of the prism $\Delta^1\times \Delta^d\simeq [\frac{1}{2},\frac{3}{4}]\times \Delta^d$. On the remaining portion $[\frac{3}{4},1]\times \B(\Q^n)$ we define $F$ by linear interpolation, since the image of $\{\frac{3}{4}\}\times \B(\Q^n)$ lies in the star of the vertex $L$ (again, it is precisely the star of $L$). Explicitly, the $d$-simplex \eqref{eq:Lsimplex} lies inside the $(d+1)$-simplex determined by
\[
0\ \lneq\ L\ \lneq\  V_0+L\ \lneq\ V_1+L\ \lneq\ \cdots\ \lneq\ V_d+L\ \lneq\ \Q^{n+1}.
\]

The map $F\colon S(\B(\Q^n))\to \B(\Q^{n+1})$ we have described is $\SL_n\Z$-equivariant by construction. We define $\phi$ to be the induced map
\[\phi\colon \St(\SL_n\Z)= H_{n-2}(\B(\Q^n);\Q)\approx H_{n-1}(S(\B(\Q^n));\Q)\overset{F_*}{\longrightarrow} H_{n-1}(\B(\Q^{n+1});\Q)= \St(\SL_{n+1}\Z),\]
and we conjecture that the map
\begin{equation}
\label{eq:SLnsteinstab}
\phi_*\colon H_i(\SL_n\Z;\St(\SL_n\Z))\to H_{i}(\SL_{n+1}\Z;\St(\SL_{n+1}\Z)).
\end{equation}
induced by $\phi$ is an isomorphism for $n> i+1$. Note that since all lines $L$ satisfying $\Z^{n+1}=\Z^n\oplus L_\Z$ are equivalent under the action of $\SL_{n+1}\Z$, the map $\phi_*$ is independent of our choice of $L$.

\begin{remarknum}
Assuming that Conjecture~\ref{conjecture:SLn} holds, the ``parabolic stabilization'' map \eqref{eq:SLnstabmap} and the ``Steinberg stabilization'' map \eqref{eq:SLnsteinstab} should be inverse to each other, and indeed one approach to Conjecture~\ref{conjecture:SLn} would be to prove this relation.
\end{remarknum}

\para{Vanishing of unstable cohomology for \boldmath$\SL_n\Z$} 
Surprisingly, the conjectured stability of \eqref{eq:SLnsteinstab} already implies that $H^{\binom{n}{2}-i}(\SL_n\Z;\Q)$ vanishes for $n> i+1$. More specifically, iterating the map \eqref{eq:SLnsteinstab} twice yields the zero map, as we now explain.
We will need the following resolution of $\St(\SL_n\Z)$, which was first written down (in slightly different form) by Ash in \cite{Ash}, following Lee--Szczarba \cite{LS}. \pagebreak

\begin{definition}[{\bf Resolution of \boldmath$\St(\SL_n\Z)$}]
\label{def:resolution}
Let
$C_k=C^n_k$ be the free $\Q$-vector space on $(n+k)$-tuples $[L_1,\ldots,L_{n+k}]$ of lines in $\Q^n$, subject to the following two relations:
\begin{itemize}
\item $[L_1,\ldots,L_{n+k}]=0$ if $\spn(L_1,\ldots,L_{n+k})\neq \Q^n$.
\item  $[L_{\sigma \cdot 1},\ldots,L_{\sigma \cdot (n+k)}] = (-1)^{|\sigma|} [L_1,\ldots,L_{n+k}]$ for $\sigma\in S_{n+k}$.
\end{itemize}
Let $C_\bullet=C^n_\bullet$ be the complex obtained by taking the standard differential $\partial\colon C_k\to C_{k-1}$:
\[\partial[L_1,\ldots,L_{n+k}]=\sum_{i=1}^{n+k} (-1)^{i-1} [L_1,\ldots,\widehat{L_i},\ldots,L_{n+k}]\]
\end{definition}
\begin{proposition}[Ash \cite{Ash}]
\label{prop:resolution}
$H_0(C^n_\bullet)\approx \St(\SL_n\Z)$ and $H_i(C^n_\bullet)=0$ for $i>0$.
\end{proposition}
Since $C^n_k$ is a virtually free $\Q\SL_n\Z$-module, Proposition~\ref{prop:resolution} states that the complex $C^n_\bullet$ is a virtually free resolution of $\St(\SL_n\Z)$. Thus the coinvariant complex $C^n_\bullet\otimes_{\Q\SL_n\Z}\Q$ computes the homology of $\SL_n\Z$ with coefficients in $\St(\SL_n\Z)$:
\[H_i\big(C^n_\bullet\otimes_{\Q\SL_n\Z}\Q\big)\approx H_i(\SL_n\Z;\St(\SL_n\Z))\]

If we choose as before a line $L$ in $\Q^{n+1}$ inducing a splitting $\Z^{n+1}=\Z^n\oplus L_\Z$, we obtain a chain map $\psi_L\colon C^n_\bullet\to C^{n+1}_\bullet$ defined by
\[
\psi_L\colon [L_1,\ldots,L_{n+k}]\mapsto [L_1,\ldots,L_{n+k},L]
\]
Since \[\spn(L_1,\ldots,L_{n+k})=\Q^n\quad\iff\quad \spn(L_1,\ldots,L_{n+k},L)= \Q^{n+1},\] $\psi_L$ preserves the first relation above. That it preserves the second is obvious, so $\psi_L$  defines a map $C^n_k\to C^{n+1}_k$. To see that $\psi_L$ commutes with $\partial$, we need only observe that $[L_1,\ldots,L_{n+k},\widehat{L}]=0$, which follows from the first relation since $\spn(L_1,\ldots,L_{n+k})\subseteq \Q^n\subsetneq \Q^{n+1}$.

On homology, it follows from \cite{Ash} that the  map $H_0(C^n_\bullet)\to H_0(C^{n+1}_\bullet)$ induced by $\psi_L$ coincides with the map $\phi\colon \St(\SL_n\Z)\to \St(\SL_{n+1}\Z)$ defined above. Thus the induced map on coinvariants \[(\psi_L)_*\colon H_i\big(C_\bullet^n \otimes_{\Q\SL_n\Z}\Q\big)\to H_i \big(C_\bullet^{n+1} \otimes_{\Q\SL_{n+1}\Z}\Q\big)\] coincides with the map on homology $\phi_*$  from \eqref{eq:SLnsteinstab} induced by $\phi$. We will use this connection to show that iterating $\phi_*$ twice yields the zero map.

Choose a line $L'$ in $\Q^{n+2}$ so that $\Z^{n+2}=\Z^n\oplus L_\Z\oplus L'_\Z$, and consider the composition $\psi_{L'}\circ \psi_L$ defined by
\[\psi_{L'}\circ \psi_L\colon [L_1,\ldots,L_{n+k}]\mapsto [L_1,\ldots,L_{n+k},L,L'].\] Let $\tau\in \SL_{n+2}\Z$ be the unique element acting by the identity on $\Q^n$ and satisfying $\tau(L)=L'$ and $\tau(L')=L$ (this element is necessarily of order 4). We have \[\tau\circ \psi_{L'}\circ \psi_L\colon [L_1,\ldots,L_{n+k}]\mapsto [L_1,\ldots,L_{n+k},L',L].\] But by the second relation in Definition~\ref{def:resolution} we have $[L_1,\ldots,L_{n+k},L',L]=-[L_1,\ldots,L_{n+k},L,L']$, and so we conclude that
\begin{equation*}
\tau\circ \psi_{L'}\circ \psi_L=-(\psi_{L'}\circ \psi_L).
\end{equation*}
Since $\tau\in \SL_{n+2}\Z$, this identity tells us that the map on coinvariants
\begin{equation}
\label{eq:coinvariantsmap}
(\phi_{L'}\circ \phi_L)_*\colon C^n_\bullet\otimes_{\Q\SL_n\Z}\Q\to C_\bullet^{n+2} \otimes_{\Q\SL_{n+2}\Z}\Q
\end{equation}
is equal to its negation, and thus is the zero map. Certainly this implies that the induced map on homology
\[\phi_*\circ \phi_*\colon H_i(\SL_n\Z;\St(\SL_n\Z))\to H_i(\SL_{n+1}\Z;\St(\SL_{n+1}\Z))\to H_{i}(\SL_{n+2}\Z;\St(\SL_{n+2}\Z))\] vanishes. This shows that if $\phi_*$ is an isomorphism for $n>i+1$, as conjectured in Conjecture~\ref{conjecture:SLn}, then $H_i(\SL_n\Z;\St(\SL_n\Z))\approx H^{\binom{n}{2}-i}(\SL_n\Z;\Q)$ must vanish for $n>i+1$, as conjectured in Conjecture~\ref{conjecture:SLnvanishing}.

\para{Congruence subgroups} The reader might wonder why we have presented Conjecture~\ref{conjecture:SLn} as a stability conjecture, if it necessarily implies the vanishing of Conjecture~\ref{conjecture:SLnvanishing}. One key reason is that this vanishing relies on torsion elements in $\SL_n\Z$\,---\,for example, our argument above depends on the fact that the order-4 element $\tau$ lies in $\SL_{n+2}\Z$. We would not expect the same vanishing if we restrict our attention to some torsion-free, finite-index subgroup of $\SL_n\Z$. However, we do expect that the stability conjectured in Conjecture~\ref{conjecture:SLn} should persist in some form.

The strongest evidence in this direction is provided by a theorem of Ash on the  level-$N$ principal congruence subgroups $\Gamma_n(N)$, meaning the subgroup of matrices in $\SL_n\Z$ reducing to the identity in $\SL_n(\Z/N\Z)$. By \cite[Theorem 3.2]{BE}, the dualizing module $\St(\Gamma_n(N))$ for the duality group $\Gamma_n(N)$ is just $\St(\SL_n\Z)$ again. Thus we may restrict the ``Steinberg stabilization'' map $\phi_*$ from \eqref{eq:SLnsteinstab} to the finite index subgroup $\Gamma_n(N)$. In this context, the main theorem of \cite{Ash} has the following form.

\begin{theorem}[Ash~\cite{Ash}] For any $N>1$, the restriction of the ``Steinberg stabilization'' map $\phi_*$ to the level-$N$ principal congruence subgroup $\Gamma_n(N)$ yields for any $n$ an \textbf{injection}
\[\phi_*\colon H^{\binom{n}{2}-i}(\Gamma_n(N);\Q)\hookrightarrow H^{\binom{n+1}{2}-i}(\Gamma_{n+1}(N);\Q).\]
\end{theorem}

\para{Cocompact lattices} The lattice $\SL_n\Z$ is not cocompact in $\SL_n\R$. However, there are natural families of cocompact lattices in $\SL_n\R$. Another reason to think of Conjecture~\ref{conjecture:SLn} as a stability conjecture is that the conjectured stability does hold for these families of cocompact lattices, as we will prove below.

Since examples of such families are not so well known, we begin by giving an explicit construction of a family of cocompact lattices $\Gamma_n$ in $\SL_n\R$ with $\Gamma_n\subset\Gamma_{n+1}$.
Let $\sqrt[4]{2}$ denote the positive real fourth root of 2.
Given $x\in \Z[\sqrt[4]{2}]$, define $||x||^2\in \Z[\sqrt{2}]$ by writing $x=a+b\sqrt[4]{2}$ for some $a,b\in \Z[\sqrt{2}]$ and defining \[||x||^2=(a+b\sqrt[4]{2})(a-b\sqrt[4]{2})=a^2-\sqrt{2}b^2.\] Define $\Gamma_n$ to be the group of matrices with entries in $\Z[\sqrt[4]{2}]$ that preserve the corresponding Hermitian form; that is, let \[\Gamma_n\coloneq \SU_n\bigl(||x_1||^2+\cdots+||x_n||^2;\Z[\sqrt[4]{2}]\bigr).\] Then $\Gamma_n$ is a cocompact lattice in $\SL_{n}(\R)$, as we now explain. The group $\Gamma_n$ is the $\Z[\sqrt{2}]$-integer points of the simple algebraic group $G$ defined over $\Q(\sqrt{2})$ given by \[G\coloneq \SU_n\bigl(||x_1||^2+\cdots+||x_n||^2;\Q(\sqrt[4]{2})\bigr).\] The group $G$ is only algebraic over $\Q(\sqrt{2})$, not over $\Q(\sqrt[4]{2})$, for the same reason that $\SU(n)$ is only a real Lie group, not a complex Lie group. A well-known theorem of Borel and Harish-Chandra (see \cite[Theorem 4.14]{PR}) states that the $\Z$-points of a semisimple algebraic group $G$ over $\Q$ form a lattice in the real points $G(\R)$. 
In our situation, the corresponding theorem states that $\Gamma_n=G(\Z[\sqrt{2}])$ is a lattice in the product $G(\R)\times G^\sigma(\R)$, where $G(\R)$ and $G^\sigma(\R)$ are obtained from $G$ by the two embeddings of $\Q(\sqrt{2})$ into $\R$ (see \cite[\S2.1.2]{PR}). For a basic example of this phenomenon, note that although $\Z[\sqrt{2}]$ is not a discrete subset of $\R$, when it is embedded in $\R\times \R$ by $a+b\sqrt{2}\mapsto (a+b\sqrt{2},a-b\sqrt{2})$ its image is discrete and indeed a lattice.

By  \cite[Proposition 2.15(3)]{PR}, $G(\R)\approx \SL_n\R$. Since the other embedding $\sigma$ sends $\sqrt{2}\mapsto -\sqrt{2}$, we have
\[G^\sigma=\SU_n(||x_1||_\sigma^2+\cdots+||x_n||_\sigma^2;\Q(\textstyle{\sqrt{-\sqrt{2}}})),\]
where $||x||^2_\sigma$ is defined by writing $x\in \Q(\sqrt{-\sqrt{2}})$ as $x=a+b\sqrt{-\sqrt{2}}$ for $a,b\in \Q(\sqrt{2})$ and defining \[||x||_\sigma^2=(a+b\textstyle{\sqrt{-\sqrt{2}}})(a-b\textstyle{\sqrt{-\sqrt{2}}})=a^2+\sqrt{2}b^2.\] It is clear from this description that when we pass from $\Q(\sqrt{2})$ to $\R$, we obtain
\[G^\sigma(\R)=\SU_n(||x_1||_\sigma^2+\cdots+||x_n||_\sigma^2;\mathbb{C})=\SU(n).\]

We conclude that $\Gamma_n$ embeds as a lattice in $\SL_n\R \times \SU(n)$. Since $\sigma(\Gamma_n)$ is a subgroup of the compact group $\SU(n)$, it contains no unipotent elements, and so neither does $\Gamma_n$. This implies that $\Gamma_n$ acts cocompactly on $\SL_n\R\times \SU(n)$ (see \cite[\S2.1.4 and Theorem 4.17(3)]{PR}). But since $\SU(n)$ is compact, the projection of $\Gamma_n$ to the first factor $\SL_n\R$ remains discrete and cocompact. We conclude that $\Gamma_n=G(\Z[\sqrt{2}])$ is a cocompact lattice in $G(\R)=\SL_n\R$.  Note that there are natural inclusions $\Gamma_n\subset\Gamma_{n+1}$ for each $n\geq 1$.

\para{Stability for cocompact lattices} In the following, $\Gamma_n$ can be any family of cocompact lattices in $\SL_n\R$, not just the explicit family described above.
Since $\Gamma_n$ is a cocompact lattice in $\SL_n\R$, it acts properly discontinuously and cocompactly on the contractible symmetric space $\SL_n\R/\SO(n)$.  By Selberg's Lemma, $\Gamma_n$ has a finite index torsion-free subgroup, which acts freely on $\SL_n\R/\SO(n)$.  Thus  
\[\vcd(\Gamma_n)=\dim \SL_n\R-\dim \SO(n)=(n^2-1)-\binom{n}{2}=\binom{n+1}{2}-1.\]
The analogue of Conjecture~\ref{conjecture:SLn} for such a family of cocompact lattices $\Gamma_n$ is the following theorem.

\begin{theorem}
For each $i\geq 0$ the group 
$H^{\binom{n+1}{2}-1-i}(\Gamma_n;\Q)$ does not depend on $n$ for $n\gg i$. 
\end{theorem}

\begin{proof} For any lattice $\Gamma_n$ in $\SL_n\R$, let $X_n$ be the locally symmetric space \[X_n\coloneq \Gamma_n\backslash \SL_n\R/\SO(n).\] Since $\Gamma_n$ acts on the contractible space $\SL_n\R/\SO(n)$ with finite stabilizers we have $H^*(X_n;\Q)\approx H^*(\Gamma_n;\Q)$. 
Since $\Gamma_n$ is cocompact, the above remarks imply that $X_n$ is a finite quotient of a closed  aspherical manifold.  Thus its rational cohomology satisfies Poincar\'e duality, which gives:
\begin{equation}
\label{eq:duality1}
H^{\binom{n+1}{2}-1-i}(\Gamma_n;\Q)\approx H^i(\Gamma_n;\Q) \ \ \text{for each $i\geq 0$}
\end{equation}

The real cohomology of the compact symmetric space $\SU(n)/\SO(n)$ is isomorphic to the space of $\SL_n\R$-invariant forms on $\SL_n\R/\SO(n)$.  These forms are closed and indeed harmonic. Being $\SL_n\R$-invariant, these forms are \emph{a fortiori} $\Gamma_n$-invariant, and so they descend to harmonic forms on $X_n$.  Thus for any lattice $\Gamma_n$ we obtain a map \[\iota\colon H^*(\SU(n)/\SO(n);\R)\to H^*(X_n;\R)\approx H^*(\Gamma_n;\R)\] If $\Gamma_n$ is cocompact, applying Hodge theory to $X_n$ implies that $\iota$ is injective in all dimensions. Moreover a theorem of 
Matsushima  \cite{Mat} implies in this case that $\iota$ is in fact surjective in a linear range of dimensions. Thus for $n\gg i$ we have for any cocompact $\Gamma_n$ (see, e.g., \cite[\S11.4]{Borel}): \[H^i(\Gamma_n;\R)\approx H^i(\SU(n)/\SO(n);\R)\approx H^i(\SU/\SO;\R)\approx \text{gr}^i\bwedge^*\langle e_5,e_9,e_{13},e_{17},\ldots\rangle\] 
In particular $H^i(\Gamma_n;\R)$ is independent of $n$ for $n\gg i$.  Applying \eqref{eq:duality1} completes the proof.
\end{proof}
We remark that Borel's proof of homological stability for $H^i(\SL_n\Z;\R)$ mentioned earlier was accomplished  by showing that $\iota$ is an isomorphism for non-cocompact lattices as well,  albeit in a smaller range of dimensions.

\para{Automorphic forms}
We close this section by briefly mentioning a connection to automorphic forms.  We recommend
\cite{BorelSurvey}, \cite{SchwermerSurvey}, and \cite[Appendix A]{SteinBook} for general surveys of
the connection between automorphic forms and the cohomology of arithmetic groups.  Generalizing
a classical result of Eichler--Shimura, Franke \cite{FrankeBorelConjecture} proved that
the groups $H^{\ast}(\SL_n\Z;\Cpx)$ are isomorphic to spaces of certain automorphic forms 
on $\SL_n\R$ (those of ``cohomological type'').  This had previously been a conjecture of Borel.  This
space of automorphic forms is the direct sum of two pieces, the cuspidal cohomology and the Eisenstein
cohomology.  However, it was observed by Borel, Wallach, and Zuckermann that the cuspidal cohomology
is all concentrated around the middle range of the cohomology (see \cite[Proposition 3.5]{SchwermerHolomorphy}
for a precise statement).  This implies that in the range described by Conjecture~\ref{conjecture:SLn},
the cohomology consists entirely of Eisenstein cohomology.  From this perspective our
conjecture is related to assertions regarding which Eisenstein series contribute to cohomology
and how Eisenstein series for different $n$ are related by induction.

\section{Stability in the unstable cohomology of mapping class groups} 
\label{section:Modg}
Let $\Mod_g$ be the mapping class group of a closed, oriented, genus $g\geq 2$ surface, and let $\M_g$ 
be the moduli space of genus $g$ Riemann surfaces.  It is well-known (see, e.g., \cite[Theorem~12.13]{FM}) that
\begin{equation}
\label{eq:ModgMg}
H^*(\Mod_g;\Q) \approx H^*(\M_g;\Q).
\end{equation}
There has been a long-standing and fruitful analogy between mapping class groups and arithmetic groups such as $\SL_n\Z$.  This
analogy is particularly strong with respect to cohomological properties, and many of the results we have described for $\SL_n\Z$ have since been proved for $\Mod_g$. Harer \cite{Ha1} proved that $H^i(\Mod_g;\Z)$ does not depend on $g$ for $g\gg i$. 
He also proved \cite{Ha2}  that $\Mod_g$ is a virtual duality group with $\nu=\vcd(\Mod_g)=4g-5$. 
Motivated by Conjecture~\ref{conjecture:SLn}, we make the following conjecture on the unstable cohomology of $\Mod_g$.
\begin{conjecture}[{\bf Stable instability}]
\label{conjecture:mcg}
For each $i\geq 0$ the group $H^{4g-5-i}(\Mod_g;\Q)$ does not depend on $g$ for $g\gg i$.
\end{conjecture}
We describe in \eqref{eq:Mess} below a stabilization map analogous to \eqref{eq:SLnstabmap} that should realize the isomorphisms conjectured in Conjecture~\ref{conjecture:mcg}. This philosophy was recently applied in \cite{CFP} to prove Conjecture~\ref{conjecture:mcg} for $i=0$ (this was also proved independently by Morita--Sakasai--Suzuki \cite{MSS} using different methods, and had been announced some years ago by Harer). However, as before, this approach has the consequence that if our conjectured stabilization map is an isomorphism for $g\gg i$,  then the ``stable unstable cohomology'' of $\Mod_g$ must vanish.
\begin{conjecture}[{\bf Vanishing Conjecture}]
\label{conjecture:mcgvanishing}
For each $i\geq 0$ we have $H^{4g-5-i}(\Mod_g;\Q)=0$ for $g\gg i$.
\end{conjecture}
Morita--Sakasai--Suzuki have pointed out in \cite[Remark 7.5]{MSS2} that Kontsevich has formulated a conjecture in \cite{Kontsevich} that would imply Conjecture~\ref{conjecture:mcgvanishing}.

\para{Computational evidence}
Complete calculations of $H^{\ast}(\Mod_g;\Q)$ are only known for $1 \leq g \leq 4$.  These
calculations are summarized in Table~\ref{table:cohomologymod}.  

\begin{table}
\begin{centering}
\begin{tabular}{l@{\hspace{26pt}}l@{\hspace{26pt}}p{\QMod}@{\hspace{5pt}}p{\QMod}@{\hspace{5pt}}p{\QMod}@{\hspace{5pt}}p{\QMod}@{\hspace{5pt}}p{\QMod}@{\hspace{5pt}}p{\QMod}@{\hspace{5pt}}p{\QMod}@{\hspace{5pt}}p{\QMod}@{\hspace{5pt}}p{\QMod}@{\hspace{5pt}}p{\QMod}@{\hspace{5pt}}p{\QMod}@{}}
\toprule
$g$ & \!\!\!$\vcd$ & \multicolumn{11}{l}{\!\!\!\!$H^1(\Mod_g;\Q)$ \hspace{5pt} $H^2(\Mod_g;\Q)$ \hspace{4pt}  $\cdots$ \hspace{4pt} $H^{\vcd}(\Mod_g;\Q)$}      \\
\midrule
$1$ & $1$  & $0$ &      &     &      &      &      &     &     &     &     &     \\
$2$ & $3$  & $0$ & $0$  & $0$ &      &      &      &     &     &     &     &     \\
$3$ & $7$  & $0$ & $\Q$ & $0$ & $0$  & $0$  & $\Q$ & $0$ &     &     &     &     \\
$4$ & $11$ & $0$ & $\Q$ & $0$ & $\Q$ & $\Q$ & $0$  & $0$ & $0$ & $0$ & $0$ & $0$ \\
\bottomrule
\end{tabular}
\caption{The rational cohomology of $\Mod_g$ for $1 \leq g \leq 4$.  For $g=1$ this is classical;
for $g=2$ this was calculated by Igusa \cite{Igusa}; for $g=3$, by
Looijenga \cite{Looijenga}; and for $g=4$, by Tommasi \cite{Tommasi}. The classes in $H^6(\Mod_3;\Q)$ and $H^5(\Mod_4;\Q)$ are unstable.}
\label{table:cohomologymod}
\end{centering}
\end{table}

\para{Mess stabilization} There is a natural analogue for $\Mod_g$ of our first stabilization map \eqref{eq:SLnstabmap} for $\SL_n\Z$. The following  topological construction provides a candidate for a stabilization map 
\begin{equation*}
H^{4(g+1)-5-i}(\Mod_{g+1};\Q)\to H^{4g-5-i}(\Mod_g;\Q)
\end{equation*}
 which could realize the  isomorphisms conjectured in Conjecture~\ref{conjecture:mcg}.
Let $S_g^1$ be a compact oriented genus $g$ surface with one boundary component
and let $\Mod_g^1$ 
be its mapping class group. Johnson proved that there is a short exact sequence
\begin{equation}
\label{eq:birman}
1\to \pi_1(T^1S_g)\to \Mod_g^1\to \Mod_g\to 1
\end{equation}
where $T^1S_g$ is the unit tangent bundle of the closed surface $S_g$. Since $T^1S_g$ is a 3--manifold and $\Mod_g$ acts trivially on $H_3(T^1S_g;\Z)$, we obtain a Gysin map $H^k(\Mod_g^1;\Z)\to H^{k-3}(\Mod_g;\Z)$.

Similarly, there is a Gysin map $H^k(\Mod_g^1\times \Z;\Z)\to H^{k-1}(\Mod_g^1;\Z)$ coming from the trivial extension \[1\to \Z\to \Mod_g^1\times \Z\to \Mod_g^1\to 1.\]
Finally, the injection $\Mod_g^1\times \Z\hookrightarrow \Mod_{g+1}$ given by sending the generator of $\Z$ to the Dehn twist $T_\delta$ around a nonseparating curve $\delta$ supported in $S_{g+1}\setminus S_g^1$ induces 
the restriction $H^k(\Mod_{g+1};\Q)\to H^k(\Mod_g^1\times \Z;\Q)$. Consider the composition:
\begin{equation}
\label{eq:MessGysin}
H^k(\Mod_{g+1};\Q)\to H^k(\Mod_g^1\times \Z;\Q)\to H^{k-1}(\Mod_g^1;\Q)\to H^{k-4}(\Mod_g;\Q)
\end{equation}
Taking $k=4(g+1)-5-i$, \eqref{eq:MessGysin} yields a map
\begin{equation}
\label{eq:Mess}
H^{4(g+1)-5-i}(\Mod_{g+1};\Q)\to H^{4g-5-i}(\Mod_g;\Q)
\end{equation}
which we conjecture is an isomorphism for $g\gg i$.
\begin{remark}
We refer to the map \eqref{eq:Mess} as \emph{Mess stabilization} because this construction was first used by Mess in \cite{Mess} to construct a subgroup $K<\Mod_g$ isomorphic to the fundamental group of a closed aspherical $(4g-5)$-manifold. This gives an explicit witness for the lower bound $\vcd(\Mod_g)\geq 4g-5$, although it follows from \cite{CFP} that the fundamental class $[K]\in H_{4g-5}(\Mod_g;\Q)$ itself vanishes rationally.
\end{remark}

\para{Vanishing of the unstable cohomology of \boldmath$\Mod_g$} We saw in \S\ref{section:SLn} that our stability conjecture for $\SL_n\Z$ necessarily implies vanishing of the cohomology in the stable range. Similarly, it turns out that if \eqref{eq:Mess} is an isomorphism then we must have $H^{4g-5-i}(\Mod_g;\Q)=0$ for $g\gg i$.

The reason is that the injection $\Mod_g^1\times \Z\hookrightarrow \Mod_{g+1}$ used in the construction of \eqref{eq:Mess} factors through the inclusion $\Mod_g^1\times_\Z \Mod_1^1\hookrightarrow \Mod_{g+1}$. This subgroup is the stabilizer of a curve $\gamma$ separating $S_{g+1}$ into two components (homeomorphic to $S_g^1$ and $S_1^1$), and the two resulting factors $\Mod_g^1$ and $\Mod_1^1$ are amalgamated over the cyclic subgroup $\langle T_\gamma\rangle\approx \Z$ generated by a Dehn twist about the separating curve $\gamma$ itself. We can write the first two maps in \eqref{eq:MessGysin} as \[H^k(\Mod_{g+1};\Q)\to H^k(\Mod_g^1\times \Z;\Q)\twoheadrightarrow H^{k-1}(\Mod_g^1;\Q)\otimes H^1(\Z;\Q)\overset{\approx}{\longrightarrow} H^{k-1}(\Mod_g^1;\Q).\]
We can factor this instead through $\Mod_g^1\times_\Z\Mod_1^1$. But the quotient $\Mod_1^1/\langle T_\gamma\rangle$ of $\Mod_1^1$ by its center is isomorphic to $\SL_2\Z$ \cite[\S2.2.4]{FM}, and so we rewrite the above as
\[H^k(\Mod_{g+1};\Q)\to H^k(\Mod_g^1\times_\Z \Mod_1^1;\Q)\to H^{k-1}(\Mod_g^1;\Q)\otimes H^1(\SL_2\Z;\Q)\to H^{k-1}(\Mod_g^1;\Q).\]
Since $H^1(\SL_2\Z;\Q)=0$, we conclude that \eqref{eq:MessGysin} is the zero map on rational cohomology.

As before, this vanishing depends on torsion phenomena (although this does not lift to torsion in $\Mod_1^1$ itself), and we expect that no such vanishing would be present if we restricted to a congruence subgroup of $\Mod_g$. The vanishing of \eqref{eq:coinvariantsmap} for $\SL_n\Z$ can be thought of as coming from the vanishing of $H^1(\SL_2\Z;\Q)$, especially in light of \cite[Main Theorem]{Ash}, and it is curious that the vanishing of our stabilization maps for mapping class groups hinges on the same fact.  Of course, if we restricted to a congruence subgroup of $\Mod_g$, the group $\SL_2\Z$ in the calculation above would be replaced by the principal congruence subgroup $\Gamma_2(N)$, and $H^1(\Gamma_2(N);\Q)\neq 0$ for any $N>1$.  

\para{Duality for \boldmath$\Mod_g$} If we could construct an analogue of the ``Steinberg stabilization'' map \eqref{eq:SLnsteinstab}, it would give a map in the other direction:
\begin{equation}
\label{eq:mcgsteinstab}
H^{4g-5-i}(\Mod_{g};\Q)\to H^{4(g+1)-5-i}(\Mod_{g+1};\Q)
\end{equation}

 Let $\C_g$ be the \emph{curve complex}, which is the simplicial complex whose $k$-simplices consist of $(k+1)$-tuples of isotopy classes of mutually disjoint simple closed curves on $S_g$.  Harer \cite[Theorem 3.5]{Ha2} proved that $\C_g$ has the homotopy type of an infinite wedge of $(2g-2)$-dimensional spheres, and the rational dualizing module for $\Mod_g$ is the Steinberg module
$\St(\Mod_g)\coloneq H_{2g-2}(\C_g;\Q).$
By definition, $\St(\Mod_g)$ satisfies
\[H^{4g-5-i}(\Mod_g;\Q)\approx H_i(\Mod_g;\St(\Mod_g))\]
for all $i\geq 0$. This gives the following equivalent formulation of Conjecture~\ref{conjecture:mcg}.

\begin{conjmcgre}
For each $i\geq 0$ the group $H_i(\Mod_g;\St(\Mod_g))$ does not depend on $g$ for $g\gg i$.
\end{conjmcgre}

\para{Proving the restated conjecture}
There are two obstructions to constructing a homomorphism 
\[H_i(\Mod_g;\St(\Mod_g))\to H_i(\Mod_{g+1};\St(\Mod_{g+1}))\] 
that could realize the conjectured isomorphisms. The first technical issue is that  there is no natural map $\Mod_g\to \Mod_{g+1}$ (or vice versa). This issue already arises in proving ordinary homological stability for $\Mod_g$. The solution there is to consider surfaces with boundary, since there is a map $\Mod_g^1\hookrightarrow\Mod_{g+1}^1$ induced by embedding $S_g^1$ into $S_{g+1}^1$. There is also a natural surjection $\Mod_g^1\twoheadrightarrow \Mod_g$ induced by gluing a disc to the boundary component of $S_g^1$. Harer proved homological stability for $\Mod_g$ by showing that both the induced maps $H_i(\Mod_g^1)\to H_i(\Mod_g)$ and $H_i(\Mod_{g}^1)\to H_i(\Mod_{g+1}^1)$ are isomorphisms for $g\gg i$.

The same tactic could be applied to our conjecture. Applying \cite[Theorem 3.5]{BE} to \eqref{eq:birman} shows that $\Mod_g^1$ is a duality group with $\nu=4g-2$ and the same dualizing module $\St(\Mod_g)$, on which $\Mod_g^1$ acts via the projection $\Mod_g^1\twoheadrightarrow \Mod_g$. Thus a first step towards the reformulation of Conjecture~\ref{conjecture:mcg} would be to prove the following.
\begin{conjecture}
\label{conjecture:forgetboundary}
For each $i\geq 0$, the natural map $H_i(\Mod_g^1;\St(\Mod_g))\longrightarrow H_i(\Mod_{g};\St(\Mod_g))$ is an isomorphism for $g\gg i$.
\end{conjecture} Since the coefficient modules are the same in this case, this seems fairly tractable. For example, for formal reasons this coincidence of coefficient modules automatically implies  Conjecture~\ref{conjecture:forgetboundary} for $i=0$  (even without the vanishing result proved in \cite{CFP} that implies that both sides are zero).

Thus if we can construct a $\Mod_g^1$--equivariant map $\St(\Mod_g)\to \St(\Mod_{g+1})$ analogous to \eqref{eq:SLnsteinstab} for $\SL_n\Z$ above, we would obtain a homomorphism
\[H_i(\Mod_g^1;\St(\Mod_g))\to H_i(\Mod_{g+1}^1;\St(\Mod_{g+1}))\] 
which combined with Conjecture~\ref{conjecture:forgetboundary} would yield the desired stabilization map of \eqref{eq:mcgsteinstab}. The most natural approach to describing such a map would be to use Broaddus' resolution of $\St(\Mod_g)$ in terms of certain pictorial \emph{chord diagrams} \cite[Prop. 3.3]{Br}, which is closely analogous to Ash's resolution of $\St(\SL_n\Z)$ from Definition~\ref{def:resolution}. However, the natural first guess for the stabilization map for $\St(\Mod_g)$ turns out to be the zero map (see \cite[Proposition~4.5]{Br})\,---\,not just on homology as occurred for $\SL_n\Z$ in \eqref{eq:coinvariantsmap}, but actually the zero map $\St(\Mod_g)\to 0\to \St(\Mod_{g+1})$. A new idea is necessary, and so we pose the following open problem.

\begin{problem}
Construct a natural nonzero $\Mod_g^1$-equivariant map $\St(\Mod_g) \rightarrow \St(\Mod_{g+1})$\linebreak analogous to the stabilization map \eqref{eq:steinmap} for $\St(\SL_n\Z)$.
\end{problem}

\section{Stability in the unstable cohomology of \boldmath$\Aut(F_n)$}
\label{section:AutFn}
The analogy between $\Mod_g$ and $\SL_n\Z$ is well-known to extend to 
the automorphism group $\Aut(F_n)$ of the  free group $F_n$ of rank $n\geq 2$.  Hatcher--Vogtmann (and later with Wahl, see \cite{HW}) 
proved that $H^i(\Aut(F_n);\Z)$ is independent of $n$ for $n\gg i$.
Culler--Vogtmann \cite{CullerVogtmann}
proved that $\vcd(\Aut(F_n))=2n-2$. 

\begin{conjecture}
\label{conjecture:AutFn}
For each $i\geq 0$ the group $H^{2n-2-i}(\Aut(F_n);\Q)$ only depends on the parity of $n$ for $n\gg i$.
\end{conjecture}
This conjecture is perhaps more speculative than Conjectures~\ref{conjecture:SLn} and \ref{conjecture:mcg}, and it remains an open question even for $i=0$. However, known conjectures on sources of unstable cohomology are consistent with Conjecture~\ref{conjecture:AutFn} for $i=1$ and $i=2$, as we explain below. 
The closely related group $\Out(F_n)$ has virtual cohomological dimension $2n-3$, and we similarly conjecture that $H^{2n-3-i}(\Out(F_n);\Q)$ only depends on the parity of $n$ for $n\gg i$.

\para{Computational evidence}
The rational cohomology groups of $\Aut(F_n)$ have been computed for $2\leq n\leq 5$, and the rational cohomology groups of $\Out(F_n)$ have been computed for $2\leq n\leq 6$. These
calculations are summarized in Table~\ref{table:cohomologyAutFn}. 
\begin{table}
\begin{centering}
\begin{tabular}{l@{\hspace{26pt}}l@{\hspace{26pt}}p{\QAut}@{\hspace{2pt}}p{\QAut}@{\hspace{2pt}}p{\QAut}@{\hspace{2pt}}p{\QAut}@{\hspace{2pt}}p{\QAut}@{\hspace{2pt}}p{\QAut}@{\hspace{2pt}}p{\QAut}@{\hspace{2pt}}p{\QAut}@{\hspace{30pt}}p{\QAut}@{\hspace{26pt}}p{\QAut}
@{\hspace{2pt}}p{\QAut}@{\hspace{2pt}}p{\QAut}@{\hspace{2pt}}p{\QAut}@{\hspace{2pt}}p{\QAut}@{\hspace{2pt}}p{\QAut}@{\hspace{2pt}}p{\QAut}
@{\hspace{2pt}}p{\QAut}@{\hspace{2pt}}p{\QAut}@{\hspace{2pt}}p{\QAut}@{\hspace{2pt}}p{\QAut}@{\hspace{2pt}}p{\QAut}@{\hspace{2pt}}p{\QAut}
@{}}
\toprule
$n$ & \!\!\!$\vcd$ & \multicolumn{8}{l}{\!\!\!\!$H^i(\Aut(F_n);\Q)$}& $\vcd$&\multicolumn{9}{l}{\!\!\!\!$H^i(\Out(F_n);\Q)$}      \\
\midrule
$2$ & $2$  & $0$ &  $0$  & & & & & & & \ 1& 0\\
$3$ & $4$  & $0$ & $0$ & $0$  &  $0$  & & & & & \ 3 & 0 & 0 & 0\\
$4$ & $6$  & $0$ & $0$ & $0$ & $\Q$ & $0$    & $0$ & & & \ 5 & 0 & 0 & 0 & $\Q$ & 0\\
$5$ & $8$  & $0$ & $0$ & $0$ & $0$ & $0$    & $0$ & $\Q$ & 0 & \ 7 & $0$ & $0$ & $0$ & $0$ & $0$    & $0$ & 0 \\
$6$ &  & & & & & & &  & & \ 9 & $0$ & $0$ & $0$ & $0$ & $0$    & $0$ & 0 & $\Q$ & 0\\
\bottomrule
\end{tabular}
\caption{The rational cohomology of $\Aut(F_n)$ for $2 \leq n \leq 5$ and of $\Out(F_n)$ for $2\leq n\leq 6$.  These were computed for $1\leq i\leq 6$ in both cases by Hatcher--Vogtmann \cite{HV}; $H^7(\Aut(F_5);\Q)$ and $H^8(\Aut(F_5);\Q)$ were computed by Gerlits (see \cite{CKV}); and $H^7(\Out(F_5);\Q)$ and $H^*(\Out(F_6);\Q)$ were computed by Ohashi \cite{Ohashi}. All rational cohomology classes are unstable.}
\label{table:cohomologyAutFn}
\end{centering}
\end{table}
\pagebreak

\para{Unstable classes and graph homology}
When $n$ is even, Morita \cite[\S6.5]{Morita} constructed cycles in $H_{2n-4}(\Out(F_n);\Q)=H_{\nu-1}(\Out(F_n);\Q)$, and Conant--Vogtmann \cite{CV} showed that these cycles can be lifted to $H_{2n-4}(\Aut(F_n);\Q)=H_{\nu-2}(\Aut(F_n);\Q)$. These classes are known to be nonzero in $H_4(\Out(F_4);\Q)$ and $H_4(\Aut(F_4);\Q)$ \cite{Morita}, in $H_8(\Out(F_6);\Q)$ and $H_8(\Aut(F_6);\Q)$ \cite{CV}, and in $H_{12}(\Out(F_8);\Q)$ and $H_{12}(\Aut(F_8);\Q)$ \cite{Gray}.  They are conjectured to be nonzero for all even $n$.

Galatius \cite{Galatius} proved that for $n\gg i$ we have $H^i(\Aut(F_n);\Q)=0$ and $H^i(\Out(F_n);\Q)=0$, so \emph{all} the rational cohomology of $\Aut(F_n)$ and $\Out(F_n)$ is unstable. The Morita cycles are known to be immediately unstable: Conant--Vogtmann \cite{CoVstab} proved that the Morita cycles vanish after stabilizing once from $H_{2n-4}(\Aut(F_n);\Q)$ to $H_{2n-4}(\Aut(F_{n+1});\Q)$. However, Conjecture~\ref{conjecture:AutFn} provides a sense in which these classes might be stable after all.

Similarly, when $n$ is odd, Conant--Kassabov--Vogtmann \cite{CKV} have recently constructed classes in $H_{2n-3}(\Aut(F_n);\Q)=H_{\nu-1}(\Aut(F_n);\Q)$, which are known to be nonzero in $H_7(\Aut(F_5);\Q)$ and $H_{11}(\Aut(F_7);\Q)$ and conjectured to be nonzero for all odd $n$. All known nonzero rational homology classes for $\Aut(F_n)$ and $\Out(F_n)$ fit into one of these families. Finally, the Morita cycles were generalized by Morita and by Conant--Vogtmann \cite[\S6.1]{CV} to produce, for every graph of rank $r$ with $k$ vertices all of odd valence, a cycle in $H_{\nu-(k-1)}(\Out(F_{r+k});\Q)$. The Morita cycles in $H_{\nu-1}(\Out(F_n);\Q)$ correspond to the graph with 2 vertices connected by $n-1$ parallel edges. Can all odd-valence graphs be naturally grouped into  families which contribute to $H_{\nu-i}(\Out(F_n);\Q)$ for some fixed $i$?

\para{Stabilization and duality}
Bestvina--Feighn \cite{BF} proved that $\Out(F_n)$ is a virtual duality group, so by \cite[Theorem 3.5]{BE} $\Aut(F_n)$ is a virtual duality group as well. The rational dualizing module $\St(\Aut(F_n))$ can be understood in terms of the topology at infinity of Culler--Vogtmann's \emph{Outer space} (see \cite[\S5]{BF} for details), but it has not been described explicitly. 

\begin{problem}
Construct a resolution for $\St(\Aut(F_n))$ analogous to Ash's resolution of $\St(\SL_n\Z)$ from Definition~\ref{def:resolution}, and analogous to Broaddus's resolution of $\St(\Mod_g))$ in terms of chord diagrams.
\end{problem}\pagebreak

\noindent We do not know an analogue for $\Aut(F_n)$ of the stabilization map
\eqref{eq:SLnsteinstab} that we constructed for $\SL_n\Z$.
\begin{problem}
Define a nontrivial, natural $\Aut(F_n)$-equivariant map $\St(\Aut(F_n))\to \St(\Aut(F_{n+1}))$.
\end{problem}
Conant--Vogtmann used Bestvina--Feighn's bordification of Outer space to construct a complex of \emph{filtered graphs} that computes the homology of $\Aut(F_n)$ \cite[\S7.3]{CVkont}.  This should yield a resolution of $\St(\Aut(F_n))$. However, from this perspective it is not clear to us how to define a stabilization map $\St(\Aut(F_n))\to \St(\Aut(F_{n+1}))$.


\para{Abelian cycles}
Consider the subgroup $K<\Aut(F_n)$ generated by $x_i\mapsto x_ix_1$ and by $x_i\mapsto x_1x_i$ for $1<i\leq n$. This subgroup is isomorphic to $\Z^{2n-2}$ and thus provides an explicit witness for the lower bound $\vcd(\Aut(F_n))\geq 2n-2$.

\begin{question}
Under the inclusion $i\colon \Z^{2n-2}\approx K\hookrightarrow \Aut(F_n)$ of the subgroup $K$, is the image of the fundamental class nonzero for some $n\geq 5$? That is, is it ever true that \[i_*[\Z^{2n-2}]\neq 0\in H_{2n-2}(\Aut(F_n);\Q)\text{?}\]
\end{question}

\para{Acknowledgements} We thank Avner Ash, Jim Conant, Matthew Emerton, Shigeyuki Morita, Akshay Venkatesh, and Karen Vogtmann for helpful conversations and correspondence. We are especially grateful to Dave Witte Morris for the description of the cocompact lattices mentioned in~\S\ref{section:SLn}.

\small

\noindent
Dept.\ of Mathematics\\
Stanford University\\
450 Serra Mall\\
Stanford, CA 94305\\
E-mail: church@math.stanford.edu
\medskip

\noindent
Dept.\ of Mathematics\\
University of Chicago\\
5734 University Ave.\\
Chicago, IL 60637\\
E-mail: farb@math.uchicago.edu
\medskip

\noindent
Dept.\ of Mathematics\\
Rice University\\
6100 Main St.\\
Houston, TX 77005\\
E-mail: andyp@rice.edu

\end{document}